\documentclass[12pt]{article}
\usepackage{mathrsfs}
\usepackage{graphicx}
\usepackage{amsmath}
\usepackage{amssymb}
\usepackage{amsthm}

\usepackage{thmtools}
\usepackage{amsfonts}

\usepackage{hyperref}
\hypersetup{
    colorlinks=true,
    linkcolor=blue,
    filecolor=blue,      
    urlcolor=cyan,
}

\usepackage{bbding}
\usepackage{CJK,CJKnumb}
\usepackage{comment}
\usepackage{mathtools}
\usepackage{xcolor}

\usepackage{etoolbox}
\patchcmd{\thebibliography}{\chapter*}{\section*}{}{}
\usepackage{ltablex}
\usepackage{float}

\usepackage[nospace,noadjust]{cite}
\usepackage{calc}
\usepackage{geometry}
\usepackage{cite}
\usepackage{tikz-cd}

\usepackage{indentfirst}
\usepackage{longtable}

\allowdisplaybreaks

\numberwithin{equation}{section} 

\newtheorem{definition}{Definition}[section]

\newtheorem{theorem}[definition]{Theorem}

\newtheorem{proposition}[definition]{Proposition}
\newtheorem{corollary}[definition]{Corollary}

\newtheorem{conjecture}[definition]{Conjecture}

\newtheorem{example}[definition]{Example}
\newtheorem*{remark}{Remark}

\newcommand{\Addresses}{{
  \bigskip
  \footnotesize

  K.~Huang, \textsc{Department of Mathematics, University of Rochester, Rochester, NY 14627}\par\nopagebreak
  \textit{E-mail address}: \texttt{keping.huang@rochester.edu}

}}

\title {Uniform Bounds for Periods of Endomorphisms of Varieties}
\author
{
Keping Huang
}

\date{\vspace{-2em} }

\begin{document}
\maketitle

\def\Z{{\bf Z}}

\begin{abstract}
Suppose $X$ is a variety defined over a finite extension $K$ of $\mathbb{Q}_p$ and suppose $X$ admits a model $\mathcal{X}$ defined over the ring of integers $R$ of $K$. Let $f:{X}\rightarrow {X}$ be an endomorphism of $X$ defined over $K$ that can be extended to an endomorphism of $\mathcal{X}$ 
defined over $R$. 
We apply a method of Fakhruddin to prove an explicit upper bound for the primitive period of periodic points defined over $R$. 
\end{abstract}

\section{Introduction and notation}

An important problem in Diophantine geometry is to calculate the number of certain types of algebraic points on varieties. 
For example, it's crucial to bound the number of rational torsion points on elliptic curves or abelian varieties.
Analogously, an important problem in arithmetic dynamics is to find the number of rational preperiodic points. 

The following Morton-Silverman Conjecture is proposed in \cite{MS94}.

\begin{conjecture}\label{unif}
Let $F/\mathbb{Q}$ be a number field of degree $D$, let 
$\phi:
\mathbb{P}^N \rightarrow \mathbb{P}^N$ be a morphism of degree $d \ge 2$ defined over $F$, 
and let $\mathrm{Prep}(\phi, F)$ be the set of $F$-rational points that are preperiodic under $\phi$. There is a constant $C(D, N, d)$
such that
$$\# \mathrm{Prep}(\phi, F) \le  C(D, N, d).$$
\end{conjecture}

If $K = \mathbb{Q}$ and we replace $\mathbb{P}^N$ by an elliptic curve $E$, then the above conjecture 
becomes Mazur's theorem proved in \cite{Maz77}, which states a uniform bound on the number of torsion rational points on elliptic curves defined over $\mathbb{Q}$. 
Mazur's Theorem has been generalized by Kamienny (\cite{Kam92}) and Kamienny-Mazur (\cite{KM95}) to quadratic fields, 
and Merel (\cite{Mer96}) to number fields of any degree.

A lot of work has been done concerning the Morton-Silverman Conjecture. 
For some of the results see \cite{Nar89, Pez94, Sil95, Li96, Zie96, Poo98, Fak01, Fak03, Hut09, BGHKST13}.


A possible approach to attack Conjecture \ref{unif} is to bound the periods of rational periodic points. 
Along this line Pezda (\cite{Pez94, Pez97}) and Zieve (\cite{Zie96}) proved bounds for the length of integral cycles of certain polynomial endomorphisms of affine spaces. 
Fakhruddin proved in \cite{Fak01} a boundedness result for endomorphisms of certain proper schemes. 
Hutz proved in \cite{Hut09} a bound for endomorphisms of non-singular projective varieties with good reduction, 
and in \cite{Hutz16} a bound for polarized endomorphisms of projective varieties. 
Bell, Ghioca, and Tucker proved in \cite{BGT15} a bound for \'etale morphisms of non-singular models of varieties.

We use the following notation throughout this paper.

\begin{center}
\begin{tabular}{p{3cm}p{9cm}}
$K$&  a finite extension of $\mathbb{Q}_p$. \\
$X/K$ & a variety of dimension $d$ defined over $K$. \\
$R$ & the ring of integer of $K$. \\
$f$ & an endomorphism of $X$ defined over $K$.  \\
$k$ & the residue field of $K$. \\
$\pi$ &  a uniformizer of $K$. \\
$v$ & the valuation of $K$ (mapping $\pi$ to 1). \\
$\mathcal{X} $ & a model of $X$. \\
$\bar{\mathcal{X}}$ & the  special fiber of $\mathcal{X}$. \\
\end{tabular}
\end{center}

The \emph{primitive period} of a periodic point $P$ has is the smallest positive integer $n$ such that $f^n(P) = P$.  The main result of this paper is the following theorem. 

\begin{theorem}\label{main}
With notation as above, 
assume that $f$ extends to an endomorphism of $\mathcal{X}$ defined over $R$. 
Let $P\in \mathcal{X}(R)$ be a periodic point under $f$. 
Let $d'$ be maximum of the dimensions of the cotangent spaces at all points in $\bar{\mathcal{X}}({k})$. 
Then the primitive period $n$ of $P$ satisfies that

\begin{equation}\label{PartTwo}
n\le \begin{cases}
|\bar{\mathcal{X}}(k)|\cdot  p^{
v(p) - 1 } \cdot \left(|k|^{d'} - 1\right),     &\text{$p > 2$},\\
|\bar{\mathcal{X}}(k)| \cdot 
2^{v(2) } \cdot \left(|k|^{d'} - 1\right),               
&\text{$p =2$}. 
\end{cases}
\end{equation}
\end{theorem}

\begin{remark}
The size $|\bar{\mathcal{X}}(k)|$ can be bounded using the Lang-Weil bound (\cite{LW54}) when $\bar{\mathcal{X}}$ is geometrically irreducible and 
the Weil bound (\cite{Wei49, Del74}) when $\bar{\mathcal{X}}$ is geometrically irreducible and non-singular. 
\end{remark}

Compared with \cite{Pez94}, \cite{Pez97}, \cite{Zie96}, \cite{Hut09}, and \cite{Hutz16}, 
our bound is larger,
but our result works for morphisms with a weaker notion of good reduction, as well as many non-projective situations, 
and our result covers all the known cases. 
We only require that the morphism $f$ can be extended to the model $\mathcal{X}$ and we do not require the model $\mathcal{X}$ to be non-singular. 
Example \ref{cubic} is a situation where the morphism $f $ does not have a good reduction, but we can still control the size of the special fiber and hence obtain a bound for the primitive period.

A key idea of this paper involves a careful analysis of the special fiber in the case of bad reduction. 
In the uniform boundedness conjecture in Diophantine geometry, 
the progress made in \cite{LT02}, \cite{KRZB16} and \cite{Sto19} 
also involved careful analysis of the special fiber in the case of bad reduction. This idea might lead to further progress in the dynamical situation.

An immediate consequence of the first part of Theorem \ref{main} is the following theorem.

\begin{theorem}\label{Qp}
With notation as before, suppose $K = \mathbb{Q}_p$ and suppose $f$ extends to an endomorphism of $\mathcal{X}$ defined over $R = \mathbb{Z}_p$. 
Let $P\in \mathcal{X}(\mathbb{Z}_p)$ be a periodic point under $f$. 
Let $d'$ be maximum of the dimensions of the cotangent spaces at all points in $\bar{\mathcal{X}}(\mathbb{F}_p)$. 
Then the primitive period $n$ of $P$ satisfies that

\begin{equation}
n\le \begin{cases}
|\bar{\mathcal{X}}(\mathbb{F}_p)| \cdot \left(p^{d'} - 1\right),     &\text{$p > 2$},\\
|\bar{\mathcal{X}}(\mathbb{F}_p) | \cdot 
2 \left(2^{d'} - 1\right),               
&\text{$p =2$}. 
\end{cases}
\end{equation}
\end{theorem}

\begin{remark}
The dimension $d'$ in the theorem is the same as the dimension of $X$ when $\mathcal{X}$ is non-singular and might be larger than the dimension of $X$ when $\mathcal{X}$ is singular. 
\end{remark}

The proof of Theorem \ref{main} follows the outline in \cite{Fak01}. 
Let $\bar{P}$ be the reduction of $P$ modulo $\pi$. 
Then $\bar{P}$ is periodic under $f$. 
First, the bound on the size $|\bar{\mathcal{X}}(k)|$ gives a bound of the period of $\bar{P}$. 
Second, passing to an iterate of $f$ we may assume the orbit of $P$ 
reduces only to one point in the special fiber $\bar{\mathcal{X}}$. Then we prove a bound for the order of the induced map on a certain cotangent space (of the model $\mathcal{X}$). 
This bound depends on the geometric data of $\bar{\mathcal{X}}$. 
This allows us to find a bound that also works for singular varieties. Third, we prove that the remaining part is a $p$-power and give a bound for it. 

The outline of this paper is as follows. 
Section \ref{weak} contains some general examples in the case of cubic polynomial morphisms over the projective line.  
In Section 3 we give the proof of Theorem \ref{main}.

\section{Weak N\'eron Models} \label{weak}
The following definition of weak N\'eron models is a slight modification of the definition on Page 278 of \cite{Hsi96}. 

\begin{definition}\label{Neron}
An $R$-scheme $\mathcal{X}$ is called a \emph{weak N\'eron model} of $(X/K, f)$ if it is separated of finite type over $R$ and 
there is a finite morphism $f': \mathcal{X} \rightarrow \mathcal{X}$ so that the following axioms hold:
\begin{enumerate}
    \item The generic fiber $\mathcal{X}_K$ of $\mathcal{X}$ is isomorphic to $X$ over $K$. 
    \item $X(K) \cong \mathcal{X}(R)$. 
    \item The restriction of the morphism $f'$ to the generic fiber of $\mathcal{X}$ is $f$. 
\end{enumerate}
\end{definition}

\begin{corollary}
If $(X/K, f)$ admits a weak N\'eron model $\mathcal{X}$, 
then any period $K$-rational point has primitive period given in Theorem \ref{main}. 
\end{corollary}

\begin{proof}
By Condition 2 of Definition \ref{Neron}, 
any $K-$point of $X$ can be extended to an $R-$point of $\mathcal{X}$. 
So we can apply Theorem \ref{main}. 
\end{proof}

\begin{example}\label{cubic}
Suppose $p>2$. 
Suppose $\phi(x)\in K[x] $ is a cubic polynomial without a $K$-rational  repelling fixed point. Then by Theorem 2.1 of \cite{BH11}, 
we know that $\phi$ admits a weak N\'eron model over $K$. 
The special fiber is either irreducible or is normal crossing with at most 2 components. 
By the proof of Theorem 2.1 of \cite{BH11} we know that the residual cycle of any periodic point $P \in \mathbb{P}^1(K)$ involves only one component of the special fiber. 
In addition, no $K$-rational point is reduced to a node in the special fiber. 
Then the primitive period $n$ of $P$ satisfies that 
\begin{equation}\label{One}
n\le
(|k| + 1)\cdot  p^{v(p) - 1} \cdot \left(|k|- 1\right). 
\end{equation}

For example, let
$\phi(x) = z + z^2 + p z^3$. 
Then $\phi$ does not have a potential good reduction by Corollary 4.6 of \cite{Ben01}. 
However, by Section 3.2 of \cite{BH11}, 
it admits no $\mathbb{Q}_p-$rational repelling fixed point. 
Then we have 
$n  \le  (p + 1)  \left(p- 1\right). $
\end{example}

\section{The Proof of Theorem \ref{main}}

Let $n_0$ be the primitive period of the reduction $\bar{P}$ of $P$. 
Consider the iterate $g:=f^{n_0}$ of $f$. Then the reduction $\bar{P}$ of $P$ is fixed under $g$. 
Let $\mathrm{Spec}(A)$ be the reduced subscheme of $\mathcal{X}$ determined by the orbit $O_g(P)$ of $P$ under $g$. 
This is possible because the fact that $P$ is periodic implies that
its $O_g(P)$ is finite, and hence the subscheme is affine. 
Recall that $A$ is local of finite rank over $R$ 
and $g$ induces an $R$-morphism $\sigma$ of $A$. 
Also let $\mathfrak{m}$ be the maximal ideal of $A$.

\begin{example}
Let $K = \mathbb{Q}_3, p=\pi = 3$ and let $X = \mathbb{P}^1$. 
Suppose $f:X\rightarrow X$ is given by $f(x) = x^2 -4x +3$ 
and $P = 0$. Then $P $ is of primitive period $2$ with orbit $O_f(P) = \{0, 3\}$, and $\bar{P}$ is fixed under $f$. 
In this case the ring $A = \mathbb{Z}_3[x]/ x(x-3)$. 
It's a local ring with maximal ideal $\mathfrak{m} = (x,3)$. 
\end{example}

We first recall 
Proposition 1 of \cite{Fak01}. We will use the notation there. 

\begin{proposition}
With the notation of Theorem \ref{main} and the above two paragraphs, 
we have $n\le n_0 rp^t$ 
where $n_0$ is the primitive period of the reduction $\bar{P}$ of $P$ 
and $r$ is the order of the induced map on the cotangent space $\mathfrak{m}/\mathfrak{m}^2$.  \qedsymbol
\end{proposition}


\begin{proof}[Proof of Theorem \ref{main}]
The proof consists of 3 steps. 

\vspace{1em}

\textbf{Step 1} Bound $n_0$.

Clearly $n_0\le |\bar{\mathcal{X}}(k)|$. 

\vspace{1em}

\textbf{Step 2} Bound $r$.

Now we show that $r$
is bounded by $|k|^{d'} - 1$. 
We first show that
$\mathrm{dim}_k (\mathfrak{m}/\mathfrak{m}^2) \le d' + 1$.

Write $\mathfrak{m} = (m,\pi)$ where $m$ is the ideal defining $P$. 
Therefore $\overline{m}:= m \pmod \pi$ defines the closed point of $\bar{A}:= A /\pi$. 
Consider the filtration 
$$\mathfrak{m} = (m,\pi) \supseteq (m^2, \pi) \supseteq \mathfrak{m}^2 = (m^2, m\pi, \pi^2). $$
As $(m,\pi)(m,\pi) \subseteq (m^2, \pi)$, we have that $(m,\pi)/(m^2, \pi)$ is a $k$-vector space and $$\mathrm{dim}_k ((m,\pi)/(m^2, \pi)) = \mathrm{dim}_k (\overline{m}/\overline{m}^2) \le d' . $$
Now we look at the quotient module $(m^2, \pi)/(m^2, m\pi, \pi^2)$. 
Clearly it's a $k$-vector space. 
Consider the map $$\phi: ( \pi)/(m\pi, \pi^2) \rightarrow (m^2, \pi)/(m^2, m\pi, \pi^2)$$
induced by the inclusion $ (\pi) \rightarrow (m^2, \pi)$. 
Since any $a\in (m^2, \pi)$ can be represented as $a = b+c $ with $b\in m^2$ and $c\in \pi$, 
we have that $\phi$ is surjective. 
It follows that 
$$\mathrm{dim}_k \left((m^2, \pi)/(m^2, m\pi, \pi^2)\right)  \le  
\mathrm{dim}_k \left(( \pi)/(m\pi, \pi^2)\right)=  \mathrm{dim}_k \left(A/(m, \pi)\right) = 1 $$
as $\pi$ is not a zero divisor in $A$. 
Therefore 
$$\mathrm{dim}_k (\mathfrak{m} / \mathfrak{m}^2)  
=  \mathrm{dim}_k \left( ( m,\pi)/(m^2,\pi) \right) 
+  \mathrm{dim}_k \left((m^2,\pi) /(m^2, m\pi, \pi^2)\right) \le d' + 1. $$

On the other hand, the subspace of $\mathfrak{m} / \mathfrak{m}^2$ 
spanned by $\pi$ is invariant under $\sigma$ as $\sigma$ is an $R$-morphism. 
Recall that $\mathfrak{m} = (m,\pi)$. 
Therefore $\sigma$ induces a $k$-linear map $\bar{\sigma}$ on $\overline{m}/\overline{m}^2$
and $\sigma$ is the identity precisely when $\bar{\sigma}$ is. 
By Corollary 2 of \cite{Dar05}, we have $ r \le |k| ^{d'}-1$
and hence the desired result.
%

\vspace{1em}

\textbf{Step 3} Bound $t$.


As in \cite{Fak01}, we look at the induced map $\sigma: A\rightarrow A$. 
Write $\sigma = \mathrm{id} + h$. 
Then $h(\mathfrak{m}) \subseteq \mathfrak{m}^2$. 
Let $\nu: A\setminus\{0\} \rightarrow \mathbb{Z}$ be defined as follows: 
for $0\neq a \in A$, let $\nu(a)$ be the largest integer $\ell$ such that
$a\in \mathfrak{m}^\ell$. 
Then $\nu = v$ on $R\setminus \{0\}$. 
Note that $\nu$ might not be additive, 
but $\nu(ar) = \nu(a) + v(r)$ for $a\in A\setminus\{0\}$ 
and $r\in R\setminus \{0\}$. 
Since $h(\mathfrak{m})\subseteq \mathfrak{m}^2$, for all $a\in \mathfrak{m}$ we must have 
either $h(a) = 0 $ or $\nu(h^j(a)) > \nu(a)~(j>0)$. 

First we show that the order $s$ of $\sigma$ is of the form $p^t$ with $t\in \mathbb{Z}_{\ge 0}$. 
It suffices to show that if $s\neq  1$ then $p|s$. 
Following the proof of Theorem 1 of \cite{Poo14}, 
we have 
\begin{equation}
    \begin{aligned}
      a &= \sigma^{s}(a) = (\mathrm{id} + h)^{s}(a) = \sum_{i=0}^{s} {{s} \choose i} h^i(a) \\
      &= a + {s \choose 1} h(a) + {s \choose 2} h^2(a) + \cdots  {s \choose {s-1}}  h^{s-1}(a) +  h^{s}(a), \\
      0 & =  {s} h(a) + {s \choose 2} h^2(a) + \dots + {s} h^{s-1}(a) +  h^{s}(a) . 
    \end{aligned}
\end{equation}
Suppose by contrary that $s \neq 1$ and $p$ does not divide $s$. Since $\nu(h^N(a)) > \nu(h(a))$ for $a\in \mathfrak{m}, N \ge 2$, we must have $0 \notin  \mathfrak{m}^{\nu(h(a)) + 1}$. Contradiction!

By Proposition 3 of \cite{Fak01} we know that $t \le v(p) - 1$ when $p>2$ and $t\le v(p)  $ when $p=2$, so we have the desired result.

\end{proof}

\section*{Acknowledgement}

I would like to express my gratitude to my advisor Thomas Tucker for suggesting this project and for many valuable discussions. I also want to thank Najmuddin Fakhruddin and Benjamin Hutz for useful comments on an earlier draft of this paper. 
Finally, I would like to thank the anonymous referee for helpful comments.

\bibliography{refFile}{}
\bibliographystyle{amsalpha}

\Addresses

\end{document}